\documentclass[12pt]{article}
\pdfoutput=1
\usepackage[T1]{fontenc}
\usepackage{iwona,amsmath,amsthm,amsfonts,amssymb,microtype,thmtools,underscore,mathtools,anyfontsize,graphicx}
\usepackage[usenames,dvipsnames,svgnames,table]{xcolor}
\usepackage[unicode=true]{hyperref}
\hypersetup{
colorlinks,
linkcolor={black},
citecolor={black},
urlcolor={blue!60!black},
pdftitle={Dominating \texorpdfstring{$K_t$}{Kt}-Models},
pdfauthor={Freddie Illingworth and David R. Wood}}
\usepackage{titlesec}
\titleformat*{\section}{\Large\bfseries}
\titleformat*{\subsection}{\large\bfseries}
\titleformat*{\subsubsection}{\bfseries}
\titleformat*{\paragraph}{\bfseries}
\titleformat*{\subparagraph}{\bfseries}
\usepackage[noabbrev,capitalise,nameinlink]{cleveref}
\crefname{lem}{Lemma}{Lemmas}
\crefname{thm}{Theorem}{Theorems}
\crefname{cor}{Corollary}{Corollaries}
\crefname{prop}{Proposition}{Propositions}
\crefname{conj}{Conjecture}{Conjectures}
\crefname{obs}{Observation}{Observations}
\crefformat{equation}{(#2#1#3)}
\Crefformat{equation}{Equation #2(#1)#3}
\usepackage[shortlabels]{enumitem}
\setlist[itemize]{topsep=0ex,itemsep=0ex,parsep=0.25ex}
\setlist[enumerate]{topsep=0ex,itemsep=0ex,parsep=0.25ex}
\newcommand{\defn}[1]{\textcolor{Maroon}{\emph{#1}}}
\usepackage[longnamesfirst,numbers,sort&compress]{natbib}
\makeatletter
\def\NAT@spacechar{~}
\makeatother
\usepackage[margin=30mm]{geometry}
\allowdisplaybreaks
\DeclarePairedDelimiter{\abs}{\lvert}{\rvert}

\DeclarePairedDelimiter{\ceil}{\lceil}{\rceil}
\DeclarePairedDelimiter{\set}{\{}{\}}
\renewcommand{\geq}{\geqslant}
\renewcommand{\leq}{\leqslant}
\renewcommand{\emptyset}{\varnothing}
\renewcommand{\epsilon}{\varepsilon}
\renewcommand{\subset}{\subseteq}
\DeclareMathOperator{\dist}{dist}

\DeclareMathOperator{\had}{had}
\DeclareMathOperator{\col}{col}
\DeclareMathOperator{\domhad}{domhad}
\newcommand{\PP}{\mathcal{P}}
\newcommand{\bE}{\mathbb{E}}
\newcommand{\bP}{\mathbb{P}}

\newcommand{\OO}{\mathcal{O}}
\renewcommand{\thefootnote}{\fnsymbol{footnote}}
\allowdisplaybreaks
\theoremstyle{plain}
\newtheorem{thm}{Theorem}
\newtheorem{lem}[thm]{Lemma}

\newtheorem{cor}[thm]{Corollary}
\newtheorem{prop}[thm]{Proposition}
\newtheorem{obs}[thm]{Observation}
\theoremstyle{definition}

\begin{document}

\author{Freddie Illingworth\,\footnotemark[1] 
 \qquad David~R.~Wood\,\footnotemark[2]}

\footnotetext[1]{Department of Mathematics, University College London, United Kingdom (\texttt{f.illingworth@ucl.ac\allowbreak.uk}). Research supported by EPSRC grant EP/V521917/1 and the Heilbronn Institute for Mathematical Research.}

\footnotetext[2]{School of Mathematics, Monash University, Melbourne, Australia  (\texttt{david.wood\allowbreak@monash.edu}). Research supported by the Australian Research Council.}
%%%%%%%

\sloppy
	
\title{\bf\boldmath Dominating $K_t$-Models}
	
\maketitle

\begin{abstract}
A \defn{dominating $K_t$-model} in a graph $G$ is a sequence $(T_1,\dots,T_t)$ of pairwise disjoint non-empty connected subgraphs of $G$, such that for $1 \leq i<j\leq t$ every vertex in $T_j$ has a neighbour in $T_i$. Replacing ``every vertex in $T_j$'' by ``some vertex in $T_j$'' retrieves the standard definition of $K_t$-model, which is equivalent to $K_t$ being a minor of $G$. We explore in what sense dominating $K_t$-models behave like (non-dominating) $K_t$-models. The two notions are equivalent for $t\leq 3$, but are already very different for $t=4$, since the 1-subdivision of any graph has no dominating $K_4$-model. Nevertheless, we show that every graph with no dominating $K_4$-model is 2-degenerate and 3-colourable. More generally, we prove that 
every graph with no dominating $K_t$-model is $2^{t-2}$-colourable. 
%%%
Motivated by the connection to chromatic number, we study the maximum average degree of graphs with no dominating $K_t$-model. We give an upper bound of $2^{t-2}$, and 
show that random graphs provide a lower bound of $(1-o(1))t\log t$, which we conjecture is asymptotically tight. This result is in contrast to the $K_t$-minor-free setting, where the maximum average degree is 
$\Theta(t\sqrt{\log t})$. 
%%%
The natural strengthening of Hadwiger's Conjecture arises: is every graph  with no dominating $K_t$-model $(t-1)$-colourable? We provide two pieces of evidence for this: (1) It is true for almost every graph, (2) Every graph $G$ with no dominating $K_t$-model has a $(t-1)$-colourable induced subgraph on at least half the vertices, which implies there is an independent set of size at least $\frac{\abs{V(G)}}{2t-2}$.
\end{abstract}

\newpage

\renewcommand{\thefootnote}{\arabic{footnote}}

\section{Introduction}

A \defn{dominating $K_t$-model} in a graph $G$ is a sequence $(T_1, \dots, T_t)$ of pairwise disjoint non-empty connected subgraphs of $G$, such that for $1 \leq i < j \leq t$ every vertex in $T_j$ has a neighbour in $T_i$. Each $T_i$ dominates $T_{i+1}\cup\dots\cup T_t$, hence the name\footnote{For a graph $G$, a set $A\subseteq V(G)$ is \defn{dominating} in $G$ if every vertex in $V(G)\setminus A$ has a neighbour in $A$. A set $A\subseteq V(G)$ is \defn{connected} if $G[A]$ is connected.}. 
Contracting each $T_i$ into a single vertex and deleting vertices not in $T_1\cup\dots\cup T_t$ gives $K_t$ as a minor. Indeed, in the definition of dominating $K_t$-model, replacing ``every vertex in $T_j$'' by ``some vertex in $T_j$'' retrieves the standard definition of $K_t$-model, which is equivalent to $K_t$ being a minor of $G$. 

At first glance, the definition of dominating $K_t$-model might seem very restrictive. This paper explores in what sense dominating $K_t$-models behave like (non-dominating) $K_t$-models. We show that several proof methods in the literature regarding $K_t$-models in fact work with dominating $K_t$-models. On the other hand, we show some significant differences. 

For $t\leq 3$ it is easily seen that a graph $G$ has a $K_t$-minor if and only if $G$ has a dominating $K_t$-model (\cref{t123}). However, for $t=4$, the behaviour changes dramatically: there are graphs that contain arbitrarily large complete graph minors, but contain no dominating $K_4$-model (\cref{Subdiv}). Nevertheless, we show that every graph with minimum degree at least 3 contains a dominating $K_4$-model (\cref{MinDegree3}), thus strengthening a classical result of \citet{Hadwiger43} and \citet{Dirac52}.

Now consider the chromatic number $\chi(G)$. \Citet{Hadwiger43} famously conjectured that every $K_t$-minor-free graph is $(t-1)$-colourable. This is widely considered to be one of the most important open problems in combinatorics (see \citep{SeymourHC} for a survey). The best upper bound on the chromatic number is $\OO(t\log\log t)$ due to \citet{DP21}. It is open whether $K_t$-minor-free graphs are $\OO(t)$-colourable.

The following natural question arises: what is the maximum chromatic number of a graph with no dominating $K_t$-model? 
We prove that this maximum exists, and in particular, the answer is at most $3\cdot 2^{t-4}$ for $t\geq 4$ (see \cref{ColouringImproved}). Complete graphs provide a lower bound of $t-1$. It is possible that every graph with no dominating $K_t$-model is $(t-1)$-colourable, which would be a considerable strengthening of Hadwiger's Conjecture. We prove this is true for $t\leq 4$ (see \cref{NoDomK4Colour}).

We show that two pieces of evidence for Hadwiger's Conjecture also hold for its dominating version. First, consider independent sets. Let \defn{$\alpha(G)$} be the maximum size of an independent set in a graph $G$. 
Hadwiger's Conjecture would imply that 
$\alpha(G)\geq \frac{n}{t-1}$ for every $n$-vertex $K_t$-minor-free graph $G$. 
\citet{DM82} proved that $\alpha(G)\geq \frac{n}{2t-2}$ for such $G$. We extend this result for graphs with no dominating $K_t$-model (\cref{BigIndSet}). 

The second piece of evidence concerns random graphs. 
The \defn{Hadwiger number} of a graph $G$, denoted by \defn{$\had(G)$}, is the maximum integer $t$ such that $K_t$ is a minor of $G$. Hadwiger's Conjecture asserts that $\chi(G)\leq\had(G)$ for every graph $G$.  Similarly, define the \defn{dominating Hadwiger number} of a graph $G$, denoted by \defn{$\domhad(G)$}, to be the maximum integer $t$ such that $G$ contains a dominating $K_t$-model.

\Citet{BCE80} showed that almost every graph $G$ satisfies $\chi(G)\leq\had(G)$. This means that the probability that a random $n$-vertex graph satisfies $\chi(G)\leq\had(G)$ tends to 1 as $n\to\infty$. Strengthening the result of \citet{BCE80}, we show that almost every graph $G$ satisfies $\chi(G)\leq\domhad(G)$ (\cref{DomHadConAlmost}). The proof of this result shows that the Hadwiger number and the dominating Hadwiger number behave differently for the Erd\H{o}s-Renyi\footnote{The \defn{Erd\H{o}s-Renyi random graph $G(n, p)$} is the $n$-vertex graph in which each edge is present with probability $p$ independently of all other edges.} random graph $G(n, p)$. In particular, 
$\domhad(G(n, \tfrac{1}{2})) = (1 + o(1))\frac{n}{\log_2 n}$ while $\had(G(n, \tfrac{1}{2})) = (1 + o(1))\frac{n}{\sqrt{\log_2 n}}$. 

Random graphs and the chromatic number both relate to the natural extremal questions: What is the maximum average degree of a $K_t$-minor-free graph?
\Citet{Kostochka82,Kostochka84} and \citet{Thomason84} independently answered this question, by showing that the maximum average degree of a $K_t$-minor-free graph is in $\Theta(t\sqrt{\log t})$, where random graphs provide the extremal examples. Later, \citet{Thomason01} determined the leading constant asymptotically. Consider the analogous question for dominating models:  What is the maximum average degree of a graph with no dominating $K_t$-model? We show that the answer is at most $2^{t-2}$ (\cref{AverageDegree}) and that random graphs provide a lower bound of $(1-o(1))t\log t$. Closing this gap is an interesting open problem. We establish a number of results that suggest a $\OO(t\log t)$ upper bound. 

%%%%%%%%%%%%%%%%%%%%%%%%%%%%%%%%%%
\section{Basics}

First note that $K_t$-models and dominating $K_t$-models are equivalent for $t\leq 3$.

\begin{obs}\label{t123}
    For $t \in \set{1,2,3}$, a graph $G$ has a dominating $K_t$-model if and only if $G$ has a $K_t$-model.
\end{obs}

\begin{proof}
    The $t=1$ case holds since a graph has a dominating $K_1$-model if and only if $V(G)\neq\emptyset$ if and only if $G$ has a $K_1$-model. 
    The $t=2$ case holds since a graph has a dominating $K_2$-model if and only if $E(G)\neq\emptyset$ if and only if $G$ has a $K_2$-model. 
    The $t=3$ case holds since a graph has dominating $K_3$-model if and only if $G$ has a cycle if and only if $G$ has a $K_3$-model. In particular, if $vw$ is an edge of a cycle $C$, then $(C-v-w, \set{v}, \set{w})$ is a dominating $K_3$-model.
\end{proof}

Note the following elementary observations. 

\begin{obs}\label{obs:domtree}
    For $t \geq 2$, a graph $G$ has a dominating $K_t$-model if and only if $G$ has a non-empty connected subgraph $T$ such that $N_G(T)$ has a dominating $K_{t - 1}$-model.
\end{obs}

\begin{obs}
    For $t\geq 2$, if a graph $G$ has a dominating $K_t$-model, then $G$ has a $K_t$-model $(T_1,\dots,T_t)$ where each of $T_{t-1}$ and $T_t$ have exactly one vertex. 
\end{obs}

\begin{proof}
    This follows from \cref{obs:domtree} and the fact that a graph has a dominating $K_2$-model if and only if it contains an edge.
\end{proof}

\begin{lem}\label{DegreePath}
    If a graph $G$ has a dominating $K_t$-model, then $G$ has a path $v_1,\dots,v_{t}$ where $\deg_G(v_t)\geq t-1$, and $\deg_G(v_i)\geq i$ for $1\leq i\leq t-1$. 
\end{lem}

\begin{proof}
    Say $(T_1,\dots,T_t)$ is a dominating $K_t$-model in $G$. Let $v_t$ be any vertex in $T_t$. Then $v_t$ has a neighbour in each of $T_1,\dots,T_{t-1}$, so $\deg_G(v)\geq t-1$, as claimed. For $i=t-1,t-2,\dots,1$, let $v_i$ be a vertex in $T_i$ adjacent to $v_{i+1}$.
    Each $v_i$ has a neighbour in each of $T_1,\dots,T_{i-1},T_{i+1}$, so $\deg_G(v_i)\geq i$. By construction, $v_1,\dots,v_{t}$ is a path. 
\end{proof}

\Cref{DegreePath} implies the following.

\begin{cor}\label{MaxDegree}
    Every graph with maximum degree $\Delta$ contains no dominating $K_{\Delta+2}$-model. 
\end{cor}

\Cref{MaxDegree} highlights a big difference between $K_t$-models and dominating $K_t$-models, since there are $n$-vertex graphs with maximum degree 3 that contain $K_{c\sqrt{n}}$-models, but contain no dominating $K_5$-model by \cref{MaxDegree}. This difference is also highlighted by the following result, which is an immediate consequence of
\cref{DegreePath}. 

\begin{cor}\label{Subdiv}
    For $n\geq 4$, if $G$ is any graph obtained from $K_n$ by subdividing each edge at least once, then $G$ contains no dominating $K_4$-model.
\end{cor}

We conclude this section by mentioning a curious consequence of \cref{MaxDegree}. If $\chi(G)\leq\domhad(G)$ for every graph $G$, then, \cref{MaxDegree} would imply that every graph with maximum degree $\Delta$ is $(\Delta+1)$-colourable, which is true since such graphs are $\Delta$-degenerate.

%%%%%%%%%%%%%%%%%%%%%%%%%%%%%%%%%%%%
\section{Dominating \texorpdfstring{$K_4$}{K4}-Models}

\citet{Hadwiger43} and \citet{Dirac52} proved that every graph with minimum degree at least 3 has a $K_4$-model. Here we strengthen this result. 

\begin{thm}\label{MinDegree3}
    Every graph $G$ with minimum degree at least 3 has a dominating $K_4$-model.
\end{thm}

\begin{proof}
    We may assume that $G$ is connected. Since $\delta(G) \geq 3$, $G$ has a cycle. Choose a pair $(C, H)$ where $C$ is a cycle of $G$ and $H$ is a component of $G - C$ such that $\abs{V(H)}$ is maximum (over all choices of $(C, H)$) and, subject to this, $\abs{V(C)}$ is minimum.
    By minimality, $C$ is an induced cycle of $G$.
    Since $\delta(G) \geq 3$, $H$ is non-empty. Since $G$ is connected, there is a vertex $z$ of $C$ with at least one neighbour in $H$.

    Suppose that $G - (V(H) \cup \set{z})$ contains a cycle $C'$. Then $G - C'$ has a component containing $\set{z} \cup V(H)$ which contradicts the choice of $(C, H)$. Thus
    \begin{equation}\label{eq:forest}
        G - (V(H) \cup \set{z}) \text{ is a forest.} \tag{$\dagger$}
    \end{equation}
    Suppose that $G - C$ has a component $H' \neq H$. If $H'$ is a single vertex $a$, then $a$ must have at least three neighbours in $C$, and so has neighbours $u, v \in C - z$. Let $P_{uv}$ be the path in $C$ from $u$ to $v$ avoiding $z$. Then $va$, $au$, and $P_{uv}$ form a cycle contradicting \eqref{eq:forest}. Thus $H'$ has at least two vertices. By \eqref{eq:forest}, $H'$ induces a tree and so has at least two leaves. Each leaf must have at least two neighbours in $C$ and so has a neighbour in $C - z$. Let $a_1$ and $a_2$ be leaves of $H'$, and $v_1, v_2 \in C - z$ be neighbours of $a_1$ and $a_2$, respectively. Let $P_{a_2 a_1}$ be the path in $H'$ from $a_2$ to $a_1$ and $P_{v_1 v_2}$ be the path in $C$ from $v_1$ to $v_2$ avoiding $z$. Then $P_{a_2 a_1}$, $a_1 v_1$, $P_{v_1 v_2}$, and $v_2 a_2$ form a cycle contradicting \eqref{eq:forest}. Thus $H$ is the only component of $G - C$.

    Since $C$ is an induced cycle of $G$ and $\delta(G) \geq 3$, every vertex of $C$ has a neighbour in $H$. Let $xy$ be an edge of $C$ and $P$ be the path $C - x - y$. Then $(H, P, \set{x}, \set{y})$ is a dominating $K_4$-model in $G$.
\end{proof}

\cref{MinDegree3} implies the following results.

\begin{cor}
\label{NoDomK4}
For $n\geq 2$, every $n$-vertex graph with no dominating $K_4$-model has at most $2n-3$ edges.
\end{cor}

\citet{Hadwiger43} and \citet{Dirac52} proved that every  $K_4$-minor-free graph is 3-colourable. \Cref{MinDegree3} implies the following strengthening. 

\begin{cor}
\label{NoDomK4Colour}
Every graph with no dominating $K_4$-model is 2-degenerate and 3-colourable. 
\end{cor}

These two results suggest that dominating $K_4$-models behave like $K_4$-models, although we emphasise that there are graphs with no dominating $K_4$-model that contain arbitrarily large complete graph minors (\cref{Subdiv}).

%%%%%%%%%%%%%%%%%%%%%%%%%%%%%%%%%%%%%%%%%%%%%%%
\section{Random Graphs}

The main result of this section (\cref{thm:domhadGnp}) asymptotically determines the dominating Hadwiger number of $G(n, p)$. We will ignore ceilings and floors in this section. We first need the following technical lemma.

\begin{lem}\label{lem:Gnpconnected}
    For fixed $p \in (0, 1)$,
    \begin{equation*}
        \bP(G(n, p) \text{ is not connected}) = \OO(n(1 - p)^n),
    \end{equation*}
    where the implied constant may depend on $p$ but not on $n$.
\end{lem}

\begin{proof}
Say a set $S$ of vertices in $G(n, p)$ is \defn{isolated} if there are no edges between $S$ and the rest of the vertices. For a fixed set $S$ of $k$ vertices,
    \begin{equation*}
        \bP(S \text{ is isolated}) = (1 - p)^{k(n - k)}.
    \end{equation*}
    If $G(n, p)$ is not connected, then there is some non-empty set of at most $n/2$ vertices that is isolated. Thus, taking a union bound and using the inequality $\binom{n}{k} \leq n^k$ gives
    \begin{align*}
        \bP(G(n, p) \text{ is not connected}) & \leq \sum_{k = 1}^{n/2} \binom{n}{k} (1 - p)^{k(n - k)} \\
        & \leq n(1 - p)^{n - 1} + n^2(1 - p)^{2(n - 2)} +  \sum_{k = 3}^{n/2} (n(1 - p)^{n - k})^k \\
        & \leq \OO(n(1 - p)^n) + \sum_{k = 3}^{n/2} (n(1 - p)^{n/2})^k \\
        & \leq \OO(n(1 - p)^n) + \frac{(n(1 - p)^{n/2})^3}{1 - n(1 - p)^{n/2}} \\
        & = \OO(n(1 - p)^n). \qedhere
    \end{align*}
\end{proof}

\begin{thm}\label{thm:domhadGnp}
    For constant $p \in (0, 1)$, asymptotically almost surely, 
    \begin{equation*}
        \domhad(G(n, p)) = (1 + o(1)) \frac{n}{\log_b n},
    \end{equation*}
    where $b = \frac{1}{1 - p}$.
\end{thm}

\begin{proof}
    We first prove the lower bound. Fix $\epsilon > 0$ and let $t = \frac{n}{(1 + \epsilon) \log_b n}$. Partition $V(G(n, p))$ into $t$ parts, $V_1, \dotsc, V_t$, each of size $n/t = (1 + \epsilon) \log_b n$. We call a part $V_i$ \defn{good} if the subgraph of $G(n, p)$ induced by $V_i$ is connected, and a pair $(V_i, V_j)$ (with $i < j$) \defn{good} if $V_j \subset N(V_i)$.

    Note that $(1-p)^{n/t}=n^{-1-\epsilon}$. 
    Since the graph induced by $V_i$ is $G(n/t, p)$,
    \cref{lem:Gnpconnected} gives
    \begin{equation*}
        \bP(V_i \text{ is bad}) = 
        \OO((1 + \epsilon) (\log_b n ) n^{-1 - \epsilon}).
    \end{equation*}
    By the union bound, the probability that some $V_i$ is bad is 
    $\OO((1 + \epsilon) (\log_b n ) n^{- \epsilon}) = o(1)$. Hence, with high probability, there are no bad parts.

    The probability that a fixed vertex in $V_j$ has no neighbour in $V_i$ is $(1 - p)^{\abs{V_i}}$. Thus,
    \begin{equation}\label{eq:pairgood}
        \bP((V_i, V_j) \text{ is good}) = (1 - (1 - p)^{\abs{V_i}})^{\abs{V_j}}.
    \end{equation}
    Using the inequality $(1 + x)^r \geq 1 + rx$ for $r \geq 1$ and $x \geq -1$ (this is Bernoulli's inequality) we obtain
    \begin{equation*}
        \bP((V_i, V_j) \text{ is bad}) = 1 - (1 - n^{-1 - \epsilon})^{n/t} \leq n^{-\epsilon} t^{-1}.
    \end{equation*}
    Hence, the expected number of bad pairs of parts is at most
    \begin{equation*}
        \tbinom{t}{2} \cdot n^{-\epsilon} t^{-1} \leq t n^{-\epsilon}.
    \end{equation*}
    By Markov's inequality, 
    the probability that there are at least $\epsilon t$ bad pairs is at most $\epsilon^{-1}n^{-\epsilon} = o(1)$. 
    Thus, with high probability, the number of bad pairs is at most $\epsilon t$ and there are no bad parts. In this case, deleting one part from each bad pair leaves a dominating clique model. Thus, with high probability, $G(n, p)$ contains a dominating $K_s$-model where
    \begin{equation*}
        s \geq (1 - \epsilon) t = \frac{1 - \epsilon}{1 + \epsilon} \cdot \frac{n}{\log_b n} \geq (1 - 2 \epsilon) \frac{n}{\log_b n}.
    \end{equation*}
    We now prove the upper bound. Define a \defn{dominating pseudo-$K_t$-model} in a graph $G$ to be a sequence $(V_1, \dotsc, V_t)$ of pairwise disjoint non-empty subsets of $V(G)$, such that $V_j \subset N(V_i)$ for $1 \leq i < j \leq t$ (this is just a dominating $K_t$-model without the connectedness condition). 
    
    Fix $\epsilon > 0$ and take $t - 1 = (1 + \epsilon) \frac{n - 1}{\log_b n}$.
    It suffices to show that $G(n, p)$ does not contain a dominating pseudo-$K_t$-model with high probability. If $(V_1, \dotsc, V_t)$ is a dominating pseudo-$K_t$-model and $v \in V_t$, then $(V(G) \setminus (V_2 \cup \dotsb \cup V_{t - 1} \cup \set{v}), V_2, \dotsc, V_{t - 1}, \set{v})$ is also a dominating pseudo-$K_t$-model. Thus it suffices to show that, with high probability, $G(n, p)$ does not contain a dominating pseudo-$K_t$-model where $\abs{V_t} = 1$ and $V_1 \cup \dotsb \cup V_t = V(G(n, p))$.

    Fix a list of non-empty vertex sets $(V_1, \dotsc, V_t)$ that partition $V(G(n, p))$ and such that $\abs{V_t} = 1$. Then $(V_1, \dotsc, V_t)$ is a dominating pseudo-$K_t$-model if and only if $(V_i,V_j)$ is good for all $1\leq i<j\leq t$. By \eqref{eq:pairgood}, and since $1-x\leq e^{-x}$ for $x\geq 0$, 
    \begin{align*}
        \bP((V_1, \dotsc, V_t) \text{ is a dominating pseudo-$K_t$-model}) & = \prod_{j = 1}^t \prod_{i = 1}^{j - 1} (1 - (1 - p)^{\abs{V_i}})^{\abs{V_j}} \\
        & \leq \prod_{j = 1}^t \prod_{i = 1}^{j - 1} \exp\set[\big]{-\abs{V_j}  (1 - p)^{\abs{V_i}}} \\
        & = \exp\set[\bigg]{-\sum_{1 \leq i < j \leq t} \abs{V_j} (1 - p)^{\abs{V_i}}}.
    \end{align*}
    We now prove a lower bound for the sum in the parentheses. Since each $V_j$ has size at least 1, 
    \begin{equation*}
        \sum_{1 \leq i < j \leq t} \abs{V_j} (1 - p)^{\abs{V_i}} \geq \sum_{1 \leq i < j \leq t} (1 - p)^{\abs{V_i}} = \sum_{i = 1}^{t - 1} (t - i) \cdot (1 - p)^{\abs{V_i}}.
    \end{equation*}
    The AM-GM inequality and the fact that $\abs{V_t} = 1$ gives
    \begin{align*}
        \sum_{i = 1}^{t - 1} (t - i) 
        (1 - p)^{\abs{V_i}} 
        & \geq (t - 1) \, 
        \biggl[\prod_{i = 1}^{t - 1} (t - i)  (1 - p)^{\abs{V_i}}\biggr]^{\frac{1}{t - 1}} \\
		& = (t - 1) \, \biggl[(t - 1)! \, (1 - p)^{\sum_{i = 1}^{t - 1} \abs{V_i}}\biggr]^{\frac{1}{t - 1}} \\
		& = (t - 1) \, [(t - 1)!]^{\frac{1}{t - 1}} \, (1 - p)^{\frac{n - 1}{t - 1}} \\
		& \geq \tfrac{(t - 1)^2}{e}\, (1 - p)^{\frac{n - 1}{t - 1}},
    \end{align*}
    where the final inequality used the lower bound $k! > \bigl(\frac{k}{e}\bigr)^k$, which is valid for all positive integers $k$. In particular,
    \begin{equation*}
        \bP((V_1, \dotsc, V_t) \text{ is a dominating pseudo-$K_t$-model}) 
        \leq 
        \exp\set[\big]{-\tfrac{(t - 1)^2}{e} \, (1 - p)^{\frac{n - 1}{t - 1}}}.
    \end{equation*}
    Since the number of partitions of $V(G(n, p))$ into $t$ parts is $t^n$, 
    \begin{align*}
        \bE(\text{\# of dominating pseudo-$K_t$-models}) & \leq t^n \exp\set[\big]{-\tfrac{(t - 1)^2}{e}\, (1 - p)^{\frac{n - 1}{t - 1}}} \\
        & = \exp\set[\big]{ (n \log t) - \tfrac{(t - 1)^2}{e}\, (1 - p)^{\frac{n - 1}{t - 1}}}.
    \end{align*}
    Recall that $t - 1 = (1 + \epsilon) \frac{n - 1}{\log_b n}$. Thus $(1-p)^{\frac{n-1}{t-1}} = n^{-1/(1+\epsilon)}$ and 
    \begin{equation*}
		n \log t - \tfrac{(t - 1)^2}{e} \, 
        (1 - p)^{\frac{n - 1}{t - 1}} 
        \leq n \log n - \tfrac{(1+\epsilon)^2}{e(\log_bn)^2} \, (n-1)^2 n^{-\frac{1}{1 + \epsilon}}.
	\end{equation*}
    Since $2 - \frac{1}{1 + \epsilon} > 1 + \frac{\epsilon}{2}$ (assuming $\epsilon < 1$), the right-hand side tends to $-\infty$ as $n \to \infty$. Thus, the expected number of dominating pseudo-$K_t$-models, and so the probability that there is a dominating pseudo-$K_t$-model is $o(1)$, as required.
\end{proof}

%%%%%%%%%%%%%%%%%%%%%%%%%%%%%%%%%%%%%%%%
\section{Structure}

This section explores the structure of graphs with no dominating $K_t$-model, and concludes a lower bound on the independence number of such graphs. The proof follows that of \citet{DM82} for $K_t$-minor-free graphs.

\begin{lem}[\citep{DM82}]
\label{DomInd}
Every non-empty connected graph has a connected dominating set $D$ and an independent set $I\subseteq D$ with $\abs{D}=2\abs{I}-1$.
\end{lem}

\begin{proof}
    Let $D$ be the largest connected set in $G$ such that $D$ contains an independent set $I$ with $\abs{D}=2\abs{I}-1$. This is well-defined since $D=I=\set{v}$ satisfies these properties for any vertex $v$. 
    Suppose for the sake of contradiction that $D$ is not dominating. Since $G$ is connected, there exists an edge $xy$ in $G$ with $\dist_G(x,D)=1$ and $\dist_G(y,D)=2$. 
    Thus $D' \coloneqq D\cup\set{x,y}$ is a connected set and $I' \coloneqq I\cup\set{y}$ is an independent set in $D'$ with $\abs{D'} = 2\abs{I'} - 1$, contradicting the choice of $D$ and $I$. Hence $D$ is dominating. 
\end{proof}

A \defn{partition} of a graph $G$ is a partition of $V(G)$ into non-empty sets. Each element of a partition is called a \defn{part}.
Let $\PP$ be a partition of a graph $G$. The \defn{quotient} of $\PP$ is the graph with vertex-set $\PP$ where distinct $A,B\in\PP$ are adjacent if and only if there is an edge of $G$ between $A$ and $B$. If $H$ is isomorphic to the quotient of $\PP$, then $\PP$ is called an \defn{$H$-partition} of $G$. We say $\PP$ is \defn{connected} if each part induces a connected subgraph of $G$, in which case the quotient is a minor of $G$. 

The \defn{depth} of a vertex $v$ in a tree rooted at $r$ is the number of vertices in the $vr$-path in $T$. The \defn{closure} of a rooted tree $T$ is the graph \defn{$\widehat{T}$} with vertex-set $V(T)$, where $vw$ is an edge of $\widehat{T}$ if and only if $v$ is an ancestor or descendent of $w$ in $T$. A graph that is isomorphic to the closure of a tree is called \defn{trivially perfect}, a \defn{comparability graph of a tree}, an \defn{arborescent comparability graph}, or a \defn{quasi-threshold graph} \citep{TPG}. They are the graphs that contain neither $P_4$ nor $C_4$ as induced subgraphs.

\begin{thm}
\label{Structure}
For every connected graph $G$ there is a tree $T$ rooted at $r$ such that $G$ has a connected $\widehat{T}$-partition $\PP$ where:
\begin{itemize}
\item for every part $P\in\PP$ there is an independent set $I$ in $G[P]$ with $\abs{P} = 2\abs{I}-1$, and
\item for every leaf $v$ of $T$, the sequence of parts in $\PP$ from $r$ to $v$ define a dominating complete graph model in $G$.
\end{itemize}
\end{thm}

\begin{proof}
We proceed by induction on $\abs{V(G)}$. The case $\abs{V(G)}=1$ is trivial. By \cref{DomInd}, $G$ has a connected dominating set $D$ and an independent set $I\subseteq D$ with $\abs{D}=2\abs{I}-1$. Let $G_1,\dots,G_c$ be the components of $G-D$. By induction, for each $i\in\set{1,\dots,c}$, there is a tree $T_i$ rooted at $r_i$ such that $G_i$  has a $\widehat{T_i}$-partition $\PP_i$ such that:
\begin{itemize}
\item for every part $P\in\PP_i$ there is an independent set $I$ in $G_i[P]$ with $\abs{P}=2\abs{I}-1$, and
\item for every leaf $v$ of $T_i$, the sequence of parts in $\PP_i$ from $r$ to $v$ define a dominating  complete graph model in $G_i$.
\end{itemize}
Let $T$ be the tree obtained from the disjoint union of $T_1,\dots,T_c$ by adding a new root vertex $r$ adjacent to $r_1,\dots,r_c$. So $\widehat{T}$ is obtained from the disjoint union of $\widehat{T_1},\dots,\widehat{T_c}$ by adding vertex $r$, adjacent to every other vertex. Let $\PP$ be the $\widehat{T}$-partition of $G$ obtained by associating $D$ with $r$. Since $D$ is dominating, for every leaf $v$ of $T$, the sequence of parts in $\PP$ from $r$ to $v$ define a dominating  complete graph model in $G$.
\end{proof}

\begin{thm}
\label{BigIndSet}
For $t \geq 2$ and every graph $G$ with no dominating $K_t$-model, there is an integer $h\leq t-1$ and an $h$-colourable induced subgraph of $G$ on at least $\frac{\abs{V(G)}+h}{2}$ vertices.
In particular, 
\begin{equation*}
    \alpha(G) \geq \tfrac{\abs{V(G)}+t-1}{2t-2}.
\end{equation*}
\end{thm}

\begin{proof}
We may assume $G$ is connected. 
Apply \cref{Structure} to $G$. Let $T$ be the resulting tree, and let $\PP$ be the resulting $\widehat{T}$-partition of $G$. 
Let $h$ be the maximum depth of a node in $T$. 
Since $G$ has no dominating $K_t$-model, $h \leq t-1$. 
Each part $A \in \PP$ contains can independent set $I_A$ of size $(\abs{A} + 1)/2$. Let $X$ be the union of the $I_A$. Then
\begin{equation*}
    \abs{X} \geq \sum_{A\in\PP}\tfrac{\abs{A}+1}{2} = \tfrac{\abs{V(G)}+\abs{\PP}}{2} \geq \tfrac{\abs{V(G)}+h}{2}.
\end{equation*}
Colour each vertex in $I_A$ by the depth of the corresponding node in $T$. We obtain a proper $h$-colouring of $X$. Taking the largest colour class in $X$, 
\begin{equation*}
    \alpha(G)\geq \tfrac{\abs{V(G)}+h}{2h} = \tfrac{\abs{V(G)}}{2h} +\tfrac{1}{2} \geq \tfrac{\abs{V(G)}}{2t-2}+\tfrac{1}{2} = 
    \tfrac{\abs{V(G)}+t-1}{2t-2}.\qedhere
\end{equation*}
\end{proof}

%%%%%%%%%%%%%%%%%%%%%%%%%%%%%%%%%%%%%%
\section{Average Degree}

This section considers the maximum average degree of a graph with no dominating $K_t$-model. \Cref{thm:domhadGnp} implies the following.

\begin{cor}\label{cor:tlogt}
    For any $\epsilon > 0$ and for sufficiently large $t$, there is a graph $G$ with no dominating $K_t$-model and average degree at least $(1 - \epsilon) t \ln t$.
\end{cor}

\begin{proof}
    Note that $\ln(\frac{1}{1 - x})/x \to 1$ as $x \to 0$. Fix $\epsilon > 0$ small and let $p > 0$ be small enough so that $\ln(\frac{1}{1 - p}) < (1 + \tfrac{\epsilon}{3}) p$. For large $n$, with high probability, $G(n, p)$ has average degree at least $(1 - \epsilon^2) pn$, and by \cref{thm:domhadGnp},
    \begin{equation*}
        \domhad(G(n,p)) \leq (1 + \tfrac{\epsilon}{3})\, \tfrac{n}{\log_b n} 
        = (1 + \tfrac{\epsilon}{3})\, \tfrac{n \ln b}{\ln n} 
        < (1 + \tfrac{\epsilon}{3})^2\, \tfrac{pn}{\ln n} 
        <  (1 + \epsilon) \,\tfrac{pn}{\ln n},
    \end{equation*}
    Pick an instance $G$ of $G(n, p)$ satisfying both these conditions. Then $G$ has no dominating $K_t$-model where $t \coloneqq (1 + \varepsilon) \frac{pn}{\log n}$. Finally, 
    \begin{equation*}
        (1 - \epsilon)\, t \ln t 
        \leq (1 - \epsilon) (1 + \epsilon) \,\tfrac{pn}{\ln n} \cdot \ln n 
        = (1 - \epsilon^2) \, pn
    \end{equation*}
    which is at most the average degree of $G$, as required.
\end{proof}

\Citet{Mader67} first showed that every graph with sufficiently high average degree contains a $K_t$-minor. The proof method actually shows the following:

\begin{thm}
\label{AverageDegree}
For $t\geq 2$ every graph with average degree at least $2^{t-2}$ has a dominating $K_t$-model.
\end{thm}

\begin{proof}
We proceed by induction on $n+t$ with the following hypothesis: for every connected $n$-vertex graph $G$ with average degree at least $2^{t-2}$ and for every vertex $v$ of $G$, there exists a dominating $K_t$-model $(T_1,\dots,T_t)$ in $G$ with $v\in V(T_1)$. This implies the claim since every graph with average degree at least $2^{t-2}$ has a connected component with average degree at least $2^{t-2}$.

The $t = 2$ case holds with $T_1 = G[\set{v}]$ and $T_2 = G[\set{w}]$ where $w$ is any neighbour of $v$ (which exists since $G$ is connected with average degree at least $1$).
If $n\leq 2^{t-2}+1$ then $G$ is complete, and the result holds trivially. 
Now assume that $t\geq 3$ and $n>2^{t-2}+1$. 
Let $G$ be a connected $n$-vertex $m$-edge graph $G$ with average degree at least $2^{t-2}$, and let $v$ be a vertex of $G$. So $2m\geq 2^{t-2}n$.

For each neighbour $w$ of $v$, let $d_w$ be the number of common neighbours of $v$ and $w$. Suppose that $d_w\leq 2^{t-3}-1$ for some neighbour $w$ of $v$. Let $G'$ be the graph obtained from $G$ by contracting $vw$ into a new vertex $v'$. So $G'$ is connected, $\abs{E(G')}=m-1-d_w$ and $\abs{V(G')}=n-1$. Thus $G'$ has average degree
\begin{equation*}
    \frac{2( m-1-d_w ) }{n-1} \geq \frac{ 2^{t-2}n - 2(1+d_w) }{n-1} \geq \frac{ 2^{t-2}n - 2^{t-2} }{n-1} = 2^{t-2}.
\end{equation*}
By induction, there exists a dominating $K_t$-model $(T_1,\dots,T_t)$ in $G'$ with $v'\in V(T_1)$. 
Let $T'_1$ be the tree obtained from $T_1$ by replacing $v'$ by $vw$, and replacing each edge $v'x$ in $T_1$ by $vx$ or $wx$ (one of which must exist). 
By construction, every vertex adjacent to $v'$ in $G'$ is adjacent to $v$ or $w$ in $G$, 
Thus  $(T'_1,T_2,\dots,T_t)$ is a dominating $K_t$-model in $G$ with $v\in V(T'_1)$, as desired. 

Now assume that every edge incident to $v$ is in at least $2^{t-3}$ triangles. Since $G$ is connected and with at least two vertices, there is at least one edge incident to $v$. Thus $G[N_G(v)]$ has minimum degree at least $2^{t-3}$. By induction, $G[N_G(v)]$ has a dominating $K_{t-1}$-model $(T_1,\dots,T_{t-1})$. Thus $(\set{v}, T_1, \dotsc, T_{t-1})$ is the desired dominating $K_{t}$-model in $G$.
\end{proof}

We conjecture that average degree $\Omega(t \log t)$ is enough to force a dominating $K_t$-model (and \cref{cor:tlogt} shows this would be best possible). As a step towards this conjecture, we show it is true for graphs with linear minimum degree. The following proofs make no attempt to optimise constants. The first lemma is well-known.

\begin{lem}
\label{FindConnDomSet}
Every connected graph $G$ with $n$ vertices and minimum degree  $\delta>\ln(2n)$ has a connected dominating set on at most $\frac{6n\ln(2n)}{\delta}$ vertices.    
\end{lem}

\begin{proof}
Let $p \coloneqq \ln(2n)/\delta<1$. For each vertex $v$, choose $v$ independently at random with probability $p$. Let $A$ be the set of chosen vertices. Let $B$ be the set of vertices $v$ such that $N_G(v)\cap A=\emptyset$. Thus $\mathbb{P}(v\in B)=(1-p)^{\deg(v)} < e^{-p\delta} = \frac{1}{2n}$. Thus $\mathbb{E}(\abs{B})<\frac{1}{2}$, and by Markov's Inequality, $\mathbb{P}(\abs{B} \geq 1) \leq \frac{\mathbb{E}(X)}{1}< \frac{1}{2}$. Now consider $\abs{A}$: $\mathbb{E}(\abs{A}) = pn$ and so, by Markov's Inequality, $\mathbb{P}(\abs{A} \geq 2pn) \leq \frac{1}{2}$. A union bound gives $\mathbb{P}(\abs{B} \geq 1 \text{ or } \abs{A}\geq 2pn) < 1$.
Hence, there exists $A \subseteq V(G)$ with $\abs{A}< 2pn = \frac{2n\ln(2n)}{\delta}$ and $B=\emptyset$. Since $B = \emptyset$, $A$ is a dominating set. 
The result follows, since \citet{DM82} showed that if a connected graph has a dominating set on $k$ vertices, then its has a connected dominating set on at most $3k-2$ vertices. 
\end{proof}

The next lemma is similar to several results in the literature on $K_t$-minors~\citep{HW16,KP08}.

\begin{lem}
\label{FindDomModel}
Fix $c\in(0,1)$ and $n$ such that $cn>\ln(2n)$. Then every graph $G$ with at most $n$ vertices and minimum degree at least $cn + 6c^{-1}t\ln(2n)$ has a dominating $K_t$-model.
\end{lem}

\begin{proof}
We proceed by induction on $t$. The case $t=1$ is trivial. 
We may assume that $G$ is connected, since any component $G'$ of $G$ has $\delta(G')\geq\delta(G)\geq cn + 6c^{-1}t\ln(2n)$ and $\abs{V(G')} \leq \abs{V(G)} \leq n$. 
Since $G$ has minimum degree at least $cn$, by \cref{FindConnDomSet}, $G$ has a connected dominating set $A$ with $\abs{A} \leq 6c^{-1}\ln(2n)$. 
Let $G' \coloneqq G-A$. So $G'$ has at most $n$ vertices and minimum degree at least 
$cn + 6c^{-1}(t-1)\ln(2n)$. 
By induction, $G'$ has a dominating $K_{t-1}$-model $(A_1,\dots,A_{t-1})$. 
Since $A$ is dominating, $(A,A_1,\dots,A_{t-1})$ is a dominating $K_t$-model in $G$. 
\end{proof}

\begin{prop}
For any $c\in(0,1)$ and sufficiently large integers $n,t$ with  $n\geq c^{-2} t\log_2 t$, every graph with $n$ vertices and minimum degree at least $2cn$ has a dominating $K_t$-model.
\end{prop}

\begin{proof}
We may assume that $t \geq (2c^{-2}\log_2 t)^6$. Thus 
\begin{equation*}
    \log_2t \geq \ln t \geq 6 \ln( 2c^{-2}\log_2 t).
\end{equation*}
Since $\frac{n}{\ln(2n)}$ is increasing and $n\geq c^{-2} t\log_2 t$, 
\begin{equation*}
    \frac{c^2n}{\ln(2n)} \geq \frac{t\log_2 t}{\ln(2 c^{-2}t\log_2 t)} \geq 6t.
\end{equation*}
Therefore
\begin{equation*}
    2cn \geq cn + 6c^{-1}t\ln(2n).
\end{equation*}
The result follows from \cref{FindDomModel} with minimum degree $2cn$, since we may assume that $cn>\ln(2n)$.
\end{proof}

Here is further evidence for the conjecture.

\begin{prop}
    For sufficiently large $t$ and $d=4t\ln t$, every $d$-regular graph $G$ contains a dominating pseudo-$K_t$-model.
\end{prop}

\begin{proof} 
Uniformly and randomly put each vertex of $G$ into one of $t$ parts. For each vertex $v$, let $B_v$ be the  event that $v$ has no neighbour in some part. Note that $\mathbb{P}(B_v)$ equals the probability that a coupon collector has not succeeded by time $4t \ln t$, which is at most $t^{-3}$ by standard tail estimate for the coupon collector problem~\citep{CouponCollector}. Note that $B_v$ depends only on the parts where the neighbours of $v$ were placed, and so $B_v$ is independent of $(B_u \colon \dist(u, v) \geq 3)$. There are at most $d^2$ vertices within distance 2 of $v$ (other than $v$). Since $4 t^{-3} d^2 < 1$ for sufficiently large $t$, the Lov\'{a}sz Local Lemma~\citep{EL75} implies that no $B_v$ occurs with positive probability. That is, there exists a partition of $V(G)$ such that every vertex in $G$ has a neighbour in every part. This partition defines a dominating pseudo-$K_t$-model.
\end{proof}

This proof still works if all degrees are close to each other, but doesn't work if the degrees vary wildly.

%%%%%%%%%%%%%%%%%%%%%%%%
\section{Colouring}

This section considers the chromatic number of graphs with no dominating $K_t$-model. \Citet{Mader67} showed that every $K_t$-minor-free graph is $2^{t-2}$-colourable. The proof generalises as follows:

\begin{thm}
\label{Colouring}
For $t\geq 2$, every graph with no dominating $K_t$-model is $2^{t-2}$-colourable. 
\end{thm}

\begin{proof}
We proceed by induction on $t$. 
Every graph $G$ with no dominating $K_2$-model has no edges, implying $G$ is $1$-colourable. Now assume that $t\geq 3$, and the result holds for $t-1$. 
Let $G$ be a graph with no dominating $K_t$-model. We may assume that $G$ is connected. Let $r$ be any vertex in $G$. For $i\geq 0$, let $V_i \coloneqq \set{v\in V(G) \colon \dist_G(v,r)=i}$. So $V_0=\set{r}$, and for $i\geq 1$, every vertex in $V_i$ has a neighbour in $V_{i-1}$. If for some $i\geq 1$, there is a dominating $K_{t-1}$-model $(T_1,\dots,T_{t-1})$ in $G[V_i]$, then $(G[V_0\cup\dots\cup V_{i-1}],T_1,\dots,T_{t-1})$ is a dominating $K_t$-model in $G$. Thus for every $i\geq 0$, there is no dominating $K_{t-1}$-model in $G[V_i]$. By induction, each $G[V_i]$ is $2^{t-3}$-colourable. Use the same set of $2^{t-3}$ colours for $\bigcup(V_i \colon i\text{ even})$ and use a disjoint  set of $2^{t-3}$ colours for $\bigcup(V_i \colon i\text{ odd})$. Since there is no edge between $V_i$ and $V_j$ with $j\geq i+2$, we obtain a $2^{t-2}$-colouring of $G$.
\end{proof}

Note that \Cref{AverageDegree} implies that every graph with no dominating $K_t$-model is $(2^{t-2}-1)$-degenerate and thus $2^{t-2}$-colourable, which provides an alternative proof of \cref{Colouring}.

The proof of \cref{Colouring} in conjunction with \cref{NoDomK4Colour} in the base case shows:

\begin{thm}
\label{ColouringImproved}
For $t\geq 4$, every graph with no dominating $K_t$-model is $3\cdot 2^{t-4}$-colourable. 
\end{thm}

Now consider the chromatic number of random graphs. For constant $p \in (0, 1)$, the asymptotic value of $\chi(G(n, p))$ was determined independently by \citet{Bollobas88} and \citet{MK90}. They proved that, with high probability,
\begin{equation}
\label{ChromaticNumberGnp}
    \chi(G(n, p)) = \bigl(\tfrac{1}{2} + o(1)\bigr) \tfrac{n}{\log_b n},
\end{equation}
where $b \coloneqq \frac{1}{1 - p}$. See \citet{Heckel18} for the most precise estimates currently known.

\citet*{BCE80} showed that for fixed $p \in (0, 1)$, with high probability,
\begin{equation*}
    \had(G(n, p)) = (1 + o(1)) \tfrac{n}{\sqrt{\log_b n}}.
\end{equation*}
It follows that Hadwiger's conjecture comfortably holds for almost every graph. 

\Cref{thm:domhadGnp,ChromaticNumberGnp} imply that:

\begin{cor}
\label{DomHadConAlmost}
Almost every graph $G$ satisfies $\chi(G) \leq \domhad(G)$.
\end{cor}

It is curious that the chromatic number and dominating Hadwiger number asymptotically differ only by a factor of 2, compared to a factor of $\sqrt{\log_2 n}$ in the case of the Hadwiger number. 

We finish this section with the following amusing observation. Let $G$ be a graph with $\chi(G) = t$. Choose a proper colouring of $G$ with colours $1, \dotsc, t$ minimising $\sum_{v\in V(G)} \col(v)$. Then for $1 \leq i < j \leq t$, every vertex of colour $j$ is adjacent to a vertex of colour $i$. In particular, the colour classes form a dominating pseudo-$K_t$-model.
This says that the `dominating pseudo-Hadwiger conjecture' is true!

%%%%%%%%%%%%%%%%%%%%%%%%
\section{Open Problems}

We finish with a number of open problems. 

\begin{itemize}
   
\item Several authors~\citep{PST03,CS12,NS22a,Blasiak07} have noted that graphs $G$ with independence number $\alpha(G)=2$ are a key unsolved case of Hadwiger's Conjecture. \Citet*{PST03} showed that Hadwiger's Conjecture holds for this class if and only if every $n$-vertex graph $G$ with $\alpha(G)=2$ has a $K_t$-minor with $t\geq \ceil{\frac{n}{2}}$. The following natural question arises: Does every $n$-vertex graph with $\alpha(G)=2$ have a dominating $K_t$-model with $t\geq \ceil{\frac{n}{2}}$? Note that if $\alpha(G)=2$ and $G$ has no dominating $K_t$-model, then $2\geq \alpha(G)\geq \frac{n+t-1}{2t-2}$ by \cref{BigIndSet}, implying $n \leq 3t - 3$. That is, if $\alpha(G)=2$ then $G$ has a dominating $K_t$-model, where $t \geq \ceil{\frac{n}{3}}$.

\item The following potential strengthening of the 4-Colour Theorem is open: Is every graph with no dominating $K_5$-model 4-colourable?

\item Does every graph with no dominating $K_t$-model have fractional chromatic number at most $2t-2$? \Citet{RS98} proved this for $K_t$-minor-free graphs.

\item Is there a rough structure theorem for graphs with no dominating $K_t$-model (in the spirit of Robertson and Seymour's rough structure theorem for $K_t$-minor-free graphs~\citep{RS-XVI} and Grohe and Marx's rough structure theorem for $K_t$-topological-minor-free graphs~\citep{GM15})? 

\end{itemize}

{\fontsize{10pt}{11pt}\selectfont

\def\soft#1{\leavevmode\setbox0=\hbox{h}\dimen7=\ht0\advance \dimen7 by-1ex\relax\if t#1\relax\rlap{\raise.6\dimen7 \hbox{\kern.3ex\char'47}}#1\relax\else\if T#1\relax \rlap{\raise.5\dimen7\hbox{\kern1.3ex\char'47}}#1\relax \else\if d#1\relax\rlap{\raise.5\dimen7\hbox{\kern.9ex \char'47}}#1\relax\else\if D#1\relax\rlap{\raise.5\dimen7 \hbox{\kern1.4ex\char'47}}#1\relax\else\if l#1\relax \rlap{\raise.5\dimen7\hbox{\kern.4ex\char'47}}#1\relax \else\if L#1\relax\rlap{\raise.5\dimen7\hbox{\kern.7ex \char'47}}#1\relax\else\message{accent \string\soft \space #1 not defined!}#1\relax\fi\fi\fi\fi\fi\fi}

}

\end{document}